\documentclass[a4paper, 10pt, notitlepage]{article}

\usepackage{amsthm, amsmath, amssymb, latexsym}
\usepackage{mathrsfs,cite}
\usepackage{graphicx}

\usepackage[ruled,vlined]{algorithm2e}

\theoremstyle{plain}
\newtheorem{theorem}{Theorem}
\newtheorem{lemma}{Lemma}

\newtheorem{proposition}{Proposition}

\theoremstyle{definition}
\newtheorem{definition}{Definition}
\newtheorem{example}{Example}

\theoremstyle{remark}
\newtheorem{remark}{Remark}

\DeclareMathOperator{\trace}{Tr}
\DeclareMathOperator*{\essinf}{ess\,inf}
\DeclareMathOperator{\interior}{int}
\DeclareMathOperator{\dist}{dist}

\author{M.V. Dolgopolik}
\title{Exact Penalty Functions with Multidimensional Penalty Parameter and Adaptive Penalty Updates}

\begin{document}

\maketitle

\begin{abstract}
We present a general theory of exact penalty functions with vectorial (multidimensional) penalty parameter for
optimization problems in infinite dimensional spaces. In comparison with the scalar case, the use of vectorial penalty
parameters provides much more flexibility, allows one to adaptively and independently take into account the violation of
each constraint during optimization process, and often leads to a better overall performance of an optimization
method using an exact penalty function. We obtain sufficient conditions for the local and global exactness of penalty
functions with vectorial penalty parameters and study convergence of global exact penalty methods with several
different penalty updating strategies. In particular, we present a new algorithmic approach to an analysis of the global
exactness of penalty functions, which contains a novel characterisation of the global exactness property in terms of
behaviour of sequences generated by certain optimization methods.
\end{abstract}

\section{Introduction}

Since their introduction by Eremin \cite{Eremin} and Zangwill \cite{Zangwill} in the mid 1960s, exact penalty functions
became one of the standard tools of constrained optimization. They are used within trust region methods
\cite{ConnGouldToint}, sequential quadratic and sequential linear-quadratic programming methods
\cite{ByrdNocedalWaltz06}, DC optimization \cite{LeThiPhamDinhNgai,Strekalovsky}, mixed-integer programming
\cite{LucidiRinaldi13}, global optimization \cite{DiPilloLucidiRinaldi}, etc. The most important
property of such penalty functions is \textit{exactness}, which allows one to reduce (locally or globally) a constrained
optimization problem to a completely equivalent unconstrained problem of minimizing an exact penalty function. Various
sufficient conditions for the local and global exactness of penalty functions were studied, e.g., in
\cite{HanMagasarian,DiPilloGrippo,RubinovYang,Zaslavski,Zaslavski2013,Demyanov2010,Dolgopolik_ExPenFunc}.

Although exact penalty functions are usually studied and applied in the finite dimensional case, some effort has been
put to analyse their behaviour for infinite dimensional problems. Sufficient conditions for the global exactness of
penalty functions in the infinite dimensional case were obtained by Zaslavski \cite{Zaslavski,Zaslavski2013} and
Demyanov \cite{Demyanov2010}. Later on, sufficient conditions from \cite{Demyanov2010} were significantly improved in
\cite{Dolgopolik_ExPenFunc,Dolgopolik_ExPenFuncII,DolgopolikFominyh}. However, existing sufficient conditions for the
global exactness of penalty functions in the infinite dimensional case impose very restrictive assumptions on
constraints of a problem under consideration (such as the Palais-Smale condition from \cite{Zaslavski,Zaslavski2013} or
conditions ensuring semiglobal metric (sub-)regularity of constraints from
\cite{Demyanov2010,Dolgopolik_ExPenFunc,Dolgopolik_ExPenFuncII,DolgopolikFominyh}), that are very hard to verify and are
not satisfied for many particular problems (cf. \cite{DolgopolikFominyh,Dolgopolik_Control}).

It should be noted that in almost all existing theoretical and applied results on penalty functions, only penalty
functions with a scalar penalty parameter (i.e. the same penalty parameter for all constraints) are considered.
Nevertheless, it is known that the use of a multidimensional/vectorial penalty parameter (i.e. the use of an individual
penalty parameter for each constraint) can noticeably improve overall performance of a numerical method based on an
exact penalty function. When a vectorial penalty parameter is used, one can adaptively adjust each individual penalty
parameter using the information on how much the corresponding constraint is violated, which can lead to better
conditioned subproblems for computing the next iterate. It might also increase the rate of convergence, since 
the violation of ``almost satisfied'' constraints is not penalized as harshly as in the case of traditional penalty
functions. However, relatively little research on penalty functions with multidimensional penalty parameter has been
done over the years. 

Lipp and Boyd \cite{LippBoyd} considered such penalty functions in the context of DC algorithm/convex-concave procedure
for cone constrained DC optimization problems, to reformulate a penalty subproblem as a convex programming problem to
which interior point methods can be applied. In the recent paper \cite{BurachikKayaPrice}, Burachick, Kaya, and Price
studied a primal-dual penalty method based on a smoothing approximation to an $\ell_1$ penalty function with vectorial
penalty parameter. The results of numerical experiments reported in \cite{BurachikKayaPrice} demonstrate that a simple 
primal-dual penalty method using vectorial penalty parameters significantly outperforms some state-of-the-art local 
optimization solvers, both in terms of computation time and quality of computed local solutions.

The main goal of this paper is to extend existing results on exact penalty functions with a single penalty
parameter to the case of penalty functions with a vectorial/multidimensional penalty parameter and present a new
\textit{algorithmic} approach to an analysis of the global exactness of penalty functions for optimization problems in
infinite dimensional spaces, which is not based on any restrictive assumptions on constraints of an optimization
problem.

In the first part of the paper, we give a general definition of penalty function with vectorial penalty parameter, 
provide some natural examples of such functions, and present several extensions of existing results on exact penalty 
functions with scalar penalty parameter to the vectorial case. In the second part of the paper, we present a new 
algorithmic approach to globally exact penalty functions for infinite
dimensional problems. Instead of imposing some restrictive (semi-)global assumptions on constraints as in
\cite{Zaslavski,Zaslavski2013,Demyanov2010,Dolgopolik_ExPenFunc,Dolgopolik_ExPenFuncII,DolgopolikFominyh}, we
demonstrate that global exactness of a penalty function can be completely characterised in terms of behaviour of
sequences generated by global exact penalty methods (namely, the existence of limit points of such sequences). We also
study behaviour of global exact penalty methods with several different types of penalty updates, including adaptive
penalty updates for penalty functions with vectorial penalty parameter (inspired by the penalty updating strategy from
\cite{BurachikKayaPrice}), which automatically adjust individual penalty parameters in accordance with the degree of
violation of corresponding constraints. We prove the global convergence of such methods, which gives one hope that
similar adaptive penalty updating strategies can be successfully used within local optimization methods based on exact
penalty functions.

The paper is organized as follows. A general definition of penalty function with vectorial (multidimensional) penalty
parameter and definitions of its local and global exactness and vectorial exact penalty parameter are given in
Section~\ref{sect:Preliminaries}. In this section we also extend standard sufficient conditions for the local exactness
to the case of penalty functions with vectorial penalty parameter and obtain necessary and sufficient conditions for
their global exactness in the finite dimensional case. New algorithmic necessary and sufficient conditions for the
global exactness of penalty functions in the infinite dimensional case, as well as a convergence analysis of related
global exact penalty methods with several penalty updating strategies, are presented in
Section~\ref{sect:GlobalExactness}.

\section{Exact Penalty Functions with Vectorial Penalty Parameter}
\label{sect:Preliminaries}

Let $(X, d)$ be a metric space, $M, Q \subseteq X$ be some sets having nonempty intersection, and 
$f \colon X \to \mathbb{R} \cup \{ + \infty \}$ be a given function. Throughout this article we study penalty functions
for the following optimization problem:
\[
  \text{minimize} \enspace f(x) \quad \text{subject to} \quad x \in M \cap Q.
  \eqno{(\mathcal{P})}
\]
Below we always suppose that there exists a globally optimal solution of this problem and the optimal value is finite.

The sets $M$ and $Q$ represent two different types of constraints of the problem $(\mathcal{P})$, e.g. equality and
inequality constraints, nonlinear and linear constraints, nonconvex and convex constraints, etc. Let us define a penalty
term for the constraint $x \in M$ and a corresponding penalty function. 

Usually, one supposes that a penalty term is a nonnegative real-valued function 
$\varphi \colon X \to \mathbb{R} \cup \{ + \infty \}$ such that $\varphi(x) = 0$ iff $x \in M$, and defines the penalty
function as $f + c \varphi$, where $c \ge 0$ is a penalty parameter
(cf.~\cite{HanMagasarian,DiPilloGrippo,RubinovYang,Zaslavski,DiPilloLucidiRinaldi}). 
Being inspired by the ideas of Lipp and Boyd \cite{LippBoyd}, we define a penalty function with vectorial penalty
parameter as follows. Let $Y$ be a real normed space, and $K \subset Y$ be a proper cone, that is, the cone $K$ is
closed, convex, and pointed (i.e. $K \cap (- K) = \{ 0 \}$). Denote by $\preceq_K$ the partial order induced by the cone
$K$, i.e. $y_1 \preceq_K y_2$ iff $y_2 - y_1 \in K$. We add improper element $\infty$ to the space $Y$, corresponding
to 
the value $+ \infty$ in the scalar case. By definition $\| \infty \| = + \infty$. 

Let $\varphi \colon X \to K \cup \{ \infty \}$ be a given function such that $\varphi(x) = 0$ iff $x \in M$. The
function $\varphi$ is called a $K$-\textit{valued} penalty term (for the constraint $x \in M$).

Let $Y^*$ be the topological dual space of $Y$, and $\langle \cdot, \cdot \rangle$ be the canonical duality pairing
between $Y^*$ and $Y$. Denote by $K^* = \{ y^* \in Y^* \mid \langle y^*, y \rangle \ge 0 \: \forall y \in K \}$ 
\textit{the dual cone} of $K$, and let $K^*_+$ be the set of all those $y^* \in K^*$ for which 
$\langle y^*, y \rangle > 0$ for all $y \in K \setminus \{ 0 \}$. Hereinafter, we suppose that $K^*_+ \ne \emptyset$.
Note that the condition $K^*_+ \ne \emptyset$ is nothing but the assumption that there exists a strictly positive 
continuous linear functional on the normed lattice $(Y, \preceq_K)$. Sufficient conditions for the existence of such
functionals can be found in \cite{vanCasteren,FernandezNaranjo}.

Choose any nonempty set $\mathcal{T} \subset K^*_+$ such that $\alpha \tau \in \mathcal{T}$ for all 
$\tau \in \mathcal{T}$ and $\alpha > 0$. For all $\tau \in \mathcal{T}$ and $x \in X$ define 
$\Phi_{\tau}(x) = f(x) + \langle \tau, \varphi(x) \rangle$. Here by definition 
$\langle \tau, \infty \rangle = + \infty$. The function $\Phi_{\tau}$ is called \textit{a penalty function}
for the problem $(\mathcal{P})$ with $\tau$ being a (vectorial) \textit{penalty parameter}. Observe that
$\Phi_{\tau}(x) \ge f(x)$ for all $x \in X$, since $\varphi(x) \in K \cup \{ \infty \}$ and $\tau \in K^*$.

For any $\tau \in \mathcal{T}$ consider the following penalized problem:
\[
  \text{minimize} \enspace \Phi_{\tau}(x) \quad \text{subject to} \quad x \in Q.
  \eqno{(\mathcal{P}_{\tau})}
\]
Observe that we incorporated the constraint $x \in M$ of the problem $(\mathcal{P})$ into the objective function of 
the penalized problem. Before we proceed to an analysis of the problem $(\mathcal{P}_{\tau})$, let us give
several simple examples illuminating the definition of penalty function with vectorial penalty parameter. Firstly,
note that putting $Y = \mathbb{R}$, $K = K^* = [0, + \infty)$, and $\mathcal{T} = K^*_+ = (0, + \infty)$ one obtains
traditional penalty functions with a single penalty parameter.

\begin{example} \label{example:MathProgram}
Consider the nonlinear programming problem
\[
  \min \: f(x) \quad \text{s.t.} \quad g_i(x) \le 0, \quad i \in I, \quad g_j(x) = 0, \quad j \in J,
  \quad x \in Q
\]
where $g_i \colon X \to \mathbb{R}$, $I = \{ 1, \ldots, m_1 \}$, and $J = \{ m_1 + 1, \ldots, m_2 \}$ for some numbers 
$m_1, m_2 \in \mathbb{N}$, $m_2 \ge m_1$. In this case
\[
  M = \Big\{ x \in X \Bigm| g_i(x) \le 0, \quad i \in I, \quad g_j(x) = 0, \quad j \in J \Big\}.
\]
Let $Y = \mathbb{R}^{m_2}$ and $K = \mathbb{R}^{m_2}_+$ be the nonnegative orthant. Then $K^*_+$ consists of all
vectors $x \in \mathbb{R}^{m_2}$ with positive coordinates, and it is natural to define $\mathcal{T} = K^*_+$. One can
set
\[
  \varphi(x) 
  = \Big( \max\{ 0, g_1(x) \}, \ldots, \max\{ 0, g_{m_1}(x) \}, |g_{m_1 + 1}(x)|, \ldots, |g_{m_2}(x)| \Big)^T
\]
for all $x \in X$. Then for any $\tau \in K^*_+$ one has
\[
  \Phi_{\tau}(x) = f(x) + \sum_{i = 1}^{m_1} \tau_i \max\{ 0, g_i(x) \}
  + \sum_{j = m_1 + 1}^{m_2} \tau_j |g_j(x)| \quad \forall x \in X,
\]
that is, $\Phi_{\tau}$ is the standard $\ell_1$ penalty function for nonlinear programming problems with individual
penalty parameter for each constraint.
\end{example}

\begin{example} \label{example:ConeConstrained}
Let $H$ be a real Hilbert space and $\mathcal{K} \subset H$ be a closed convex cone. Denote by 
$\mathcal{K}^{\circ} = \{ h \in H \mid \langle h, z \rangle \le 0 \: \forall z \in \mathcal{K} \}$ 
\textit{the polar cone} of $\mathcal{K}$. Consider the cone constrained optimization problem
\begin{equation} \label{eq:ConeConstrProblem}
  \min \: f(x) \quad \text{s.t.} \quad G(x) \in \mathcal{K}, \quad x \in Q,
\end{equation}
where $G \colon X \to H$ is a given function. Define $M = \{ x \in X \mid G(x) \in \mathcal{K} \}$, $Y = H$, and 
$K = \mathcal{K}^{\circ}$. In this case $K^*_+$ consists of all those $y \in -\mathcal{K}$ for which 
$\langle y, z \rangle > 0$ for all $z \in K \setminus \{ 0 \}$. We set $\mathcal{T} = K^*_+$.

By the well-known Moreau theorem \cite{Moreau,Soltan}, a point $y \in Y$ belongs to $\mathcal{K}$ iff the metric
projection of $y$ onto the cone $K = \mathcal{K}^{\circ}$, denoted by $Pr_K(y)$, is zero. Therefore one can define
$\varphi(x) = Pr_K(G(x))$ for all $x \in X$ and, if $K^*_+ \ne \emptyset$, define
\[
  \Phi_{\tau}(x) = f(x) + \langle \tau, Pr_K(G(x)) \rangle \quad \forall x \in X, \: \tau \in K^*_+.
\]
In particular, let $H = \mathbb{S}^{\ell}$ be the space of real symmetric matrices of order $\ell \in \mathbb{N}$
equipped with the inner product $\langle A, B \rangle = \trace(A B)$ and the corresponding norm, which is called the
Frobenius norm (here $\trace(A)$ is the trace of a matrix $A$). Let $\mathcal{K} = \mathbb{S}^{\ell}_-$ be the cone of
negative semidefinite matrices. Then problem \eqref{eq:ConeConstrProblem} is the standard nonlinear semidefinite
programming problem of the form
\[
  \min \: f(x) \quad \text{s.t.} \quad G(x) \preceq 0, \quad x \in Q,
\]
where the relation $G(x) \preceq 0$ means that the matrix $G(x) \in \mathbb{S}^{\ell}$ is negative semidefinite. In this
case $K = \mathbb{S}^{\ell}_+$ is the cone of positive semidefinite matrices, while $K^*_+ = \mathcal{T}$ is the cone
of positive definite matrices. Furthermore, one has
\[
  \varphi(x) = [G(x)]_+, \quad \Phi_{\tau}(x) = f(x) + \trace\big( \tau [G(x)]_+ \big) \quad \forall x \in X,
\]
where $[A]_+$ is the metric projection of a matrix $A$ onto the cone of positive semidefinite matrices. In this case
vectorial penalty parameter $\tau$ is any positive definite matrix (cf.~\cite{LippBoyd}).
\end{example}

\begin{example} \label{example:OptimalControl}
Consider the following optimal control problem with a pointwise state constraint:
\begin{align*}
  &\text{minimize} \enspace \mathcal{I}(x, u) = \int_0^T L(x(t), u(t), t) \, dt 
  \\ 
  &\text{subject to} \quad \dot{x}(t) = F(x(t), u(t), t), \quad u(t) \in U(t) \quad \text{for a.e. } t \in [0, T], 
  \\
  &x(0) = x_0, \quad x(T) = x_T, \quad g(x(t), t) \le 0 \quad \forall t \in [0, T]
\end{align*}
(for the sake of shortness, we suppose that there is only one state constraint). Here $x \colon [0, T] \to \mathbb{R}^d$
belongs to the space of absolutely continuous on $[0, T]$ vector-valued functions $AC([0, T]; \mathbb{R}^d)$ and 
$u \in L^{\infty}([0, T]; \mathbb{R}^m)$, while $L$, $F$, and $g$ are continuous functions. We would like to convert
this problem to an optimal control problem without state constraints via penalty functions. To this end, define
$X = AC([0, T]; \mathbb{R}^d) \times L^{\infty}([0, T]; \mathbb{R}^m)$, and introduce the sets
$M = \{ (x, u) \in X \mid g(x(t), t) \le 0 \enspace \forall t \in [0, T] \}$ and
\begin{multline*}
  Q = \Big\{ (x, u) \in X \Bigm| \dot{x}(t) = F(x(t), u(t), t), \enspace u(t) \in U(t) \text{ for a.e. } t \in [0, T],
  \\
  x(0) = x_0, \quad x(T) = x_T \Big\}.
\end{multline*}
Let $Y = C[0, T]$ be the space of continuous functions, and $K \subset Y$ be the cone of nonnegative functions. Then by
the Reisz-Markov-Kakutani theorem, the dual cone $K^*$ can be identified with the set of regular Borel measures on 
$[0, T]$, while the set $K^*_+$ contains, in particular, all regular Borel measures that are absolutely continuous with
respect to the Lebesgue measure and have a.e. positive density. We denote by $\mathcal{T}$ the set of such Borel
measures, which can obviously be identified with the set of a.e. positive Lebesgue integrable functions 
$\tau \colon [0, T] \to (0, + \infty)$.

Define $\varphi(x, u) = \max\{ 0, g(x(t), t) \}$. Then
\[
  \Phi_{\tau}(x, u) = \mathcal{I}(x, u) + \int_0^T \tau(t) \max\{ 0, g(x(t), t) \} dt \quad \forall (x, u) \in X
\]
for any positive Lebesgue integrable function $\tau \in \mathcal{T}$. In this example, the use of the set $\mathcal{T}$
instead of $K^*_+$ allows one to exclude irregular penalty parameters (i.e. Borel measures with nonzero discrete and/or
singular parts) from consideration.
\end{example}

Let us now turn to an analysis of the exactness properties of the penalty function $\Phi_{\tau}$. We start by extending
the definition of local exactness (cf.~\cite{HanMagasarian,Dolgopolik_ExPenFunc}) to the case of penalty functions with
vectorial penalty parameter. Recall that $\preceq_{K^*}$ is the partial order induced by the cone $K^*$, that is,
$\tau_1 \preceq_{K^*} \tau_2$ if and only if $\langle \tau_1, y \rangle \le \langle \tau_2, y \rangle$ for all $y \in
K$.

\begin{definition}
Let $x_*$ be a locally optimal solution of the problem $(\mathcal{P})$. The penalty function $\Phi_{\tau}$ is called
\textit{locally exact} at the point $x_*$, if there exists $\tau_* \in \mathcal{T}$ such that for all 
$\tau \in \mathcal{T}$ satisfying the condition $\tau \succeq_{K^*} \tau_*$, the point $x_*$ is a locally optimal
solution of the penalized problem $(\mathcal{P}_{\tau})$. Any such $\tau_*$ is called 
\textit{a local exact penalty parameter} at $x_*$.
\end{definition}

If the penalty function $\Phi_{\tau}$ is locally exact at $x_*$, then by definition there exists 
$\tau_* \in \mathcal{T}$ and a neighbourhood $U$ of $x_*$ such that $\Phi_{\tau_*}(x) \ge \Phi_{\tau_*}(x_*)$ 
for all $x \in U \cap Q$. Observe that $\Phi_{\tau}(x_*) = f(x_*)$ for all $\tau \in \mathcal{T}$ due to the fact that
$x_*$ is a feasible point of the problem $(\mathcal{P})$, i.e. $\varphi(x_*) = 0$. Therefore, for 
any $\tau \succeq_{K^*} \tau_*$ one has
\[
  \Phi_{\tau}(x) \ge \Phi_{\tau_*}(x) \ge f(x) = \Phi_{\tau}(x_*) \quad \forall x \in U \cap Q.
\]
Note that the neighbourhood $U$ is the same for all $\tau$. In other words, for any $\tau \succeq_{K^*} \tau_*$ the
point $x_*$ is a local, \textit{uniformly} with respect to $\tau \succeq_{K^*} \tau_*$, minimizer of the penalized
problem $(\mathcal{P}_{\tau})$.

Let us extend standard sufficient conditions for the local exactness of penalty functions
(cf.~\cite{HanMagasarian,Dolgopolik_ExPenFunc}) to the vectorial case, by showing that the metric subregularity of
constraints along with the Lipschitz continuity of the objective function guarantees the local
exactness of the penalty function $\Phi_{\tau}$.

Let $K^*_{++}$ be the set of all those $\tau \in K^*_+$ for which there exists $c > 0$ such that
$\langle \tau, y \rangle \ge c \| y \|$ for all $y \in K$. The supremum of all such $c$ is denoted by
$p_K(\tau)$. The function $p_K(\cdot)$ is obviously positively homogeneous. Moreover, the following equality
holds true:
\[
  p_K(\tau) = \inf\Big\{ \langle \tau, y \rangle \Bigm| y \in K, \: \| y \| = 1 \Big\}.
\]
Note that if the cone $K$ is finite dimensional, then this infimum is attained and positive for any 
$\tau \in K^*_+$, since in this case the set $\{ y \in K \mid \| y \| = 1 \}$ is compact and 
$\langle \tau, y \rangle > 0$ for any vector $y$ from this set. Thus, $K^*_{++} = K^*_+$ in the case when the cone $K$
is finite dimensional. However, in the general case this equality does not hold true. In particular, in
Example~\ref{example:OptimalControl} the set $\mathcal{T} \cap K^*_{++}$ consists of all Lebesgue integrable functions
$\tau \colon [0, T] \to (0, + \infty)$ with positive essential infimum on $[0, T]$, and 
$p_K(\tau) = \essinf_{t \in [0, T]} \tau(t)$.

\begin{proposition} \label{prp:LocalExactness}
Let $\mathcal{T} \cap K^*_{++} \ne \emptyset$, $x_*$ be a locally optimal solution of the problem $(\mathcal{P})$, and
$f$ be H\"{o}lder continuous with constant $L > 0$ and exponent $\alpha > 0$ near $x_*$. Suppose also that there exist
$\eta > 0$ and a neighbourhood $U$ of $x_*$ such that
\begin{equation} \label{eq:LocalErrorBound}
  \| \varphi(x) \| \ge \eta \Big( \dist(x, M \cap Q) \Big)^{\alpha} \quad \forall x \in U \cap Q.
\end{equation}
Then the penalty function $\Phi_{\tau}$ is locally exact at $x_*$ 
with local exact penalty parameter $(L / \eta p_K(\tau)) \tau$ for any $\tau \in \mathcal{T} \cap K^*_{++}$.
\end{proposition}

\begin{proof}
Denote $\psi(x) = \| \varphi(x) \|$. Then by definition $\psi(x) = 0$ iff $x \in M$, and 
$\psi(x) \ge \eta (\dist(x, M \cap Q))^{\alpha}$ for all $x \in U \cap Q$. Therefore, by
\cite[Thrm.~2.4 and Prp.~2.7]{Dolgopolik_ExPenFunc} the penalty function $f + c \psi$ is exact at $x_*$ with exact
penalty parameter $L / \eta$. Hence with the use of the inequality $\Phi_{\tau}(x) \ge f + p_K(\tau) \psi$ we arrive
at the required result.
\end{proof}

\begin{remark}
Suppose that in Example~\ref{example:ConeConstrained} $X$ is a Banach space and the function $G$ is continuously
Fr\'{e}chet differentiable at a locally optimal solution $x_*$ of the problem $(\mathcal{P})$. Then, as is well-known,
the validity of Robinson's constraint qualification
\[
  0 \in \interior\Big\{ G(x_*) + DG(x_*)(Q - x_*) - \mathcal{K} \Big\}
\]
implies that inequality \eqref{eq:LocalErrorBound} with $\alpha = 1$ holds true (see, e.g.
\cite[Crlr.~2.2]{Cominetti}). Thus, in this case the penalty function $\Phi_{\tau}$ is locally exact at $x_*$,
provided Robinson's constraint qualification holds at $x_*$ and the objective function $f$ is Lipschitz continuous near
this point. More generally, it is sufficient to suppose that the multifunction 
\[
  H(x) = \begin{cases}
    G(x) - \mathcal{K}, & \text{if } x \in Q, \\
    \emptyset, & \text{if } x \notin Q
  \end{cases}
\]
is metrically subregular near $(x_*, 0)$ and $f$ is Lipschitz continuous near $x_*$.
\end{remark}

Let us now consider globally exact penalty functions.

\begin{definition} \label{def:GlobalExactness}
The penalty function $\Phi_{\tau}$ is said to be \textit{globally exact}, if there exists $\tau_* \in \mathcal{T}$ such
that for all $\tau \succeq_{K^*} \tau_*$ the set of globally optimal solutions of the problem $(\mathcal{P})$ coincides
with the set of globally optimal solutions of the penalized problem $(\mathcal{P}_{\tau})$. Any such $\tau_*$ is
called a (global) \textit{exact penalty parameter}.
\end{definition}

It should be noted that instead of verifying that the sets of globally optimal solutions of the problems
$(\mathcal{P})$ and $(\mathcal{P}_{\tau})$ coincide, it is sufficient to check that these problems have the same
optimal value.

\begin{lemma} \label{lemma:ExactnessViaOptValue}
The penalty function $\Phi_{\tau}$ is globally exact iff there exists $\tau_* \in \mathcal{T}$ such that the optimal
value of the problem $(\mathcal{P})$ coincides with the optimal value of the problem $(\mathcal{P}_{\tau_*})$. Moreover,
any $\tau \succeq_{K^*_+} \tau_*$ (i.e. $\tau - \tau_* \in K^*_+$) is an exact penalty parameter of $\Phi_{\tau}$.
\end{lemma}

\begin{proof}
Bearing in mind the fact that for any feasible point (in particular, globally optimal solution) $x$ of the problem
$(\mathcal{P})$ one has $\varphi(x) = 0$ and $\Phi_{\tau}(x) = f(x)$, one gets that if the sets of globally optimal
solutions of the problems $(\mathcal{P})$ and $(\mathcal{P}_{\tau})$ coincide, then the optimal values of these
problems coincide as well.

Let us prove the converse statement. Suppose that for some $\tau_* \in \mathcal{T}$ the optimal values of the problems
$(\mathcal{P})$ and $(\mathcal{P}_{\tau_*})$ coincide. Then, in particular, for any globally optimal solution $x_*$ of
the problem $(\mathcal{P})$ and for any $\tau \succeq_{K^*} \tau_*$ one has
\[
  \inf_{x \in Q} \Phi_{\tau}(x) \ge \inf_{x \in Q} \Phi_{\tau_*}(x) = f(x_*) = \Phi_{\tau}(x_*),
\]
which implies that for any $\tau \succeq_{K^*} \tau_*$ the point $x_*$ is a globally optimal solution of the problem
$(\mathcal{P}_{\tau})$. On the other hand, for any $\tau \succeq_{K^*_+} \tau_*$ (e.g. for $\tau = 2 \tau_*$) and any
point $x$, that is infeasible for the problem $(\mathcal{P})$, one has $\varphi(x) \ne 0$ and
\[	
  \Phi_{\tau}(x) > \Phi_{\tau_*}(x) \ge \Phi_{\tau_*}(x_*),
\]
i.e. for any $\tau \succeq_{K^*_+} \tau_*$ globally optimal solutions of the problem $(\mathcal{P}_{\tau})$ must be
feasible for the problem $(\mathcal{P})$. Hence with the use of the fact that for any feasible point $x$ of the
problem $(\mathcal{P})$ and for any $\tau \in K^*$ one has $\Phi_{\tau}(x) = f(x)$, one can conclude that for all
$\tau \succeq_{K^*_+} \tau_*$ globally optimal solutions of the problems $(\mathcal{P})$ and
$(\mathcal{P}_{\tau})$ coincide, i.e. any such $\tau$ is an exact penalty parameter.
\end{proof}

Let us now turn to an analysis of necessary and/or sufficient conditions for the global exactness of the penalty
function $\Phi_{\tau}$. At first, let us point out an almost trivial, yet useful comparison principle that allows one to
prove the local/global exactness of a penalty function with vectorial penalty parameter by proving the local/global
exactness of the corresponding standard penalty function with scalar penalty parameter (cf. the proof of 
Proposition~\ref{prp:LocalExactness}). With the use of this principle one can apply existing conditions for the global
exactness of penalty functions to the penalty function $\Phi_{\tau}$. For the sake of shortness, we formulate the
comparison principle only for globally exact penalty functions.

\begin{lemma}[Comparison Principle for Penalty Functions]
Let $\mathcal{T} \subseteq K^*_{++}$. Then the penalty function $\Phi_{\tau}$ is globally exact if and only if the
penalty function $\Psi_c(\cdot) = f(\cdot) + c \| \varphi(\cdot) \|$, $c > 0$, is globally exact. Moreover, if $\tau_*$
is a global exact penalty parameter for $\Phi_{\tau}$, then $c_* = \| \tau_* \|_{Y^*}$, where $\| \tau_* \|_{Y^*}$ is
the
norm of $\tau_*$ in $Y^*$, is a global exact penalty parameter for $\Psi_c$. Conversely, if $c_*$ is a global exact
penalty parameter for $\Psi_c$, then any $\tau \in \mathcal{T}$ with $p_K(\tau) \ge c_*$ is a global exact penalty
parameter for $\Phi_{\tau}$.
\end{lemma}

\begin{proof}
Observe that
\[
  \Psi_c(x) \le f(x) + p_K(\tau) \| \varphi(x) \| \le \Phi_{\tau}(x) 
  \le f(x) + \| \tau \|_{Y^*} \| \varphi(x) \| \le \Psi_s(x) 
\]
for all $x \in X$, $c \le p_K(\tau)$, and $s \ge \| \tau \|_{Y^*}$. Moreover, $\Psi_c(x) = \Phi_{\tau}(x) = f(x)$
for any feasible point $x$ of the problem $(\mathcal{P})$, and for all $c \ge 0$ and $\tau \in K^*$. Therefore, if the
set of global minimizers of $\Psi_c$ on $Q$ coincides with the set of globally optimal solutions of the problem
$(\mathcal{P})$ for some $c > 0$, then so does the set of global minimizers of $\Phi_{\tau}$ on $Q$ for any
$\tau \in \mathcal{T}$ with $p_K(\tau) \ge c$. Similarly, if the set of global minimizers of $\Phi_\tau$ on $Q$
coincides with the set of globally optimal solutions of the problem $(\mathcal{P})$ for some $\tau \in \mathcal{T}$,
then so does the set of global minimizers of $\Psi_c$ on $Q$ for any $c \ge \| \tau \|_{Y^*}$. Hence taking into
account Def.~\ref{def:GlobalExactness} one obtains the required result.
\end{proof}

\begin{remark} \label{rmrk:PenaltyParameterMultiple}
From the comparison principle it follows that if the penalty function $\Phi_{\tau}$ is globally exact, then for any
$\tau \in \mathcal{T} \cap K^*_{++}$ there exists $c(\tau) > 0$ such that $c(\tau) \tau$ is a global exact penalty
parameter of $\Phi_{\tau}$. Indeed, if $\Phi_{\tau}$ is globally exact, then by the comparison principle the penalty
function $\Psi_c = f(\cdot) + c \| \varphi(\cdot) \|$ is globally exact as well. Let $c_* > 0$ be its global exact
penalty parameter. Then applying the comparison principle once again one obtains that for any 
$\tau \in \mathcal{T} \cap K^*_{++}$ and $c(\tau) \ge c_* / p_K(\tau)$ the vector $c(\tau) \tau$ is a global exact
penalty parameter of $\Phi_{\tau}$. The same statement obviously holds true for local exact penalty parameters.
\end{remark}

In the end of this section we obtain necessary and sufficient conditions for the global exactness of the penalty
function $\Phi_{\tau}$ in the finite dimensional case in the form of the so-called localization principle. Roughly
speaking, this principle states that the global exactness of a penalty function is completely defined by its local
behaviour near globally optimal solutions of the problem under consideration. Various versions of the localization
principle for exact penalty functions with a single penalty parameter and augmented Lagrangian functions were studied in
detail in \cite{Dolgopolik_Unified,Dolgopolik_UnifiedII}. Let $f_*$ be the optimal value of the problem $(\mathcal{P})$.

\begin{theorem}[Localization Principle]
Let $X$ be a finite dimensional normed space, the set $Q$ be a closed, and the functions $f$ and $\| \varphi(\cdot) \|$
be lower semicontinuous (l.s.c.) on $Q$. Suppose also that $\mathcal{T} \subseteq K^*_{++}$. Then the penalty function
$\Phi_{\tau}$ is globally exact if and only if the two following conditions hold true:
\begin{enumerate}
\item{$\Phi_{\tau}$ is locally exact at every globally optimal solution of the problem $(\mathcal{P})$;
}

\item{there exists $\tau_0 \in \mathcal{T}$ such that the sublevel set
\begin{equation} \label{eq:SublevelSet}
  \Big\{ x \in Q \Bigm| \Phi_{\tau_0}(x) < f_* \Big\}
\end{equation}
is either bounded or empty.
}
\end{enumerate}
\end{theorem}

\begin{proof}
If $\Phi_{\tau}$ is globally exact with exact penalty parameter $\tau_*$, then it is obviously locally exact at every
globally optimal solution of the problem $(\mathcal{P})$ with the same exact penalty parameter. Furthermore, by
Lemma~\ref{lemma:ExactnessViaOptValue} the sublevel set \eqref{eq:SublevelSet} is empty for $\tau_0 = \tau_*$.

Let us prove the converse statement. Since $\Phi_{\tau}$ is locally exact at every globally optimal solution of 
the problem $(\mathcal{P})$, by the comparison principle the penalty function 
$\Psi_c(\cdot) = f(\cdot) + c \| \varphi(\cdot) \|$ is locally exact at every globally optimal solution of the problem
$(\mathcal{P})$ as well. Moreover, from the inequality
\[
  \Phi_{\tau_0}(x) \le f(x) + \| \tau_0 \|_{Y^*} \| \varphi(x) \| \le \Psi_c(x) 
  \quad \forall x \in X, \: c \ge \| \tau_0 \|_{Y^*}
\]
and the second assumption of the theorem it follows that for any $c \ge \| \tau_0 \|_{Y^*}$ the sublevel set 
$\{ x \in Q \mid \Psi_c(x) < f_* \}$ is either bounded or empty. Therefore, by the localization principle for linear
exact penalty functions \cite[Thrm.~3.1]{Dolgopolik_Unified}, the penalty function $\Psi_c$ is globally exact, which by
the comparison principle implies that the penalty function $\Phi_{\tau}$ is globally exact as well.
\end{proof}

\begin{remark}
Note that in the localization principle we do \textit{not} make any assumptions on local exact penalty parameters of
$\Phi_{\tau}$ at globally optimal solutions of the problem $(\mathcal{P})$. Even if there is an infinite number of such
solutions, no assumptions on the corresponding local exact penalty parameters (such as boundedness) are needed to prove
the localization principle. Let us also note that one can guarantee the boundedness of the sublevel set
\eqref{eq:SublevelSet} by assuming that either the set $Q$ is bounded or the function 
$f(\cdot) + c \| \varphi(\cdot) \|$ is coercive on the set $Q$ for some $c > 0$, i.e. 
$f(x_n) + c \| \varphi(x_n) \| \to + \infty$ as $n \to + \infty$, if $\{ x_n \} \subset Q$ and $\| x_n \| \to + \infty$
as $n \to \infty$.
\end{remark}

\section{An Algorithmic Approach to Global Exactness and Adaptive Penalty Updates}
\label{sect:GlobalExactness}

As one might expect, the localization principle for exact penalty functions does not hold true in the infinite
dimensional case (see \cite[Examples~3--5]{Dolgopolik_ExPenFunc}). In order to prove the global exactness of a penalty
function for infinite dimensional problems, one usually must impose some very restrictive assumptions on constraints,
that are not satisfied in many particular examples 
(cf. \cite{Zaslavski,Zaslavski2013,Demyanov2010,Dolgopolik_ExPenFunc,DolgopolikFominyh,Dolgopolik_Control}). In this
section, we 
present a completely new \textit{algorithmic} approach to an analysis of the global exactness of penalty functions in
the infinite dimensional case. This approach is based on an analysis of behaviour of minimization sequences generated
by global exact penalty methods. It allows one to obtain simple necessary and sufficient conditions for the global
exactness of penalty functions, that do not rely on restrictive assumptions on constraints and are much more suitable
for design and analysis of exact penalty methods than existing conditions.

Let us first prove the following auxiliary result on behaviour of global minimizers of $\Phi_{\tau}$ as the penalty
parameter goes to infinity (cf. analogous results for standard penalty functions, e.g.
\cite[Prp.~3.5]{Dolgopolik_ExPenFunc}).

\begin{lemma} \label{lemma:ExPenMinimizingSeq}
Let $\tau \in \mathcal{T} \cap K^*_{++}$ be given, $\{ c_n \} \subset (0, + \infty)$ be a strictly increasing unbounded
sequence, and $x_n$ be a point of global minimum of the function $\Phi_{\tau_n}$ on the set $Q$ for any 
$n \in \mathbb{N}$, where $\tau_n = c_n \tau$. Then the sequence $\{ f(x_n) \}$ is nondecrasing and $\varphi(x_n) \to 0$
as $n \to \infty$. If, in addition, $Q$ is closed and both $f$ and $\| \varphi(\cdot) \|$ are l.s.c. on $Q$, then all
limit points of the sequence $\{ x_n \}$ (if exist) are globally optimal solutions of the problem $(\mathcal{P})$.
\end{lemma}

\begin{proof}
Let us first show that the sequence $\{ f(x_n) \}$ is nondecreasing. Indeed, fix any $n \in \mathbb{N}$. Then by
definition
\begin{align*}
  f(x_{n + 1}) + c_{n + 1} \langle \tau, \varphi(x_{n + 1}) \rangle &= \Phi_{\tau_{n + 1}}(x_{n + 1})
  \\
  &\le \Phi_{\tau_{n + 1}}(x_n) = f(x_n) + c_{n + 1} \langle \tau, \varphi(x_n) \rangle,
\end{align*}
which yield
\begin{equation} \label{eq:ObjFuncDecay}
  f(x_{n + 1}) - f(x_n) \le c_{n + 1} \langle \tau, \varphi(x_n) - \varphi(x_{n + 1}) \rangle.
\end{equation}
Similarly, by definition one has
\[
  f(x_n) + c_n \langle \tau, \varphi(x_n) \rangle = \Phi_{\tau_n}(x_n)
  \le \Phi_{\tau_n}(x_{n + 1}) = f(x_{n + 1}) + c_n \langle \tau, \varphi(x_{n + 1}) \rangle,
\]
which implies that
\begin{equation} \label{eq:ObjFuncDecay2}
  f(x_n) - f(x_{n + 1}) \le c_n \langle \tau, \varphi(x_{n + 1}) - \varphi(x_n) \rangle.
\end{equation}
Adding \eqref{eq:ObjFuncDecay} and \eqref{eq:ObjFuncDecay2} one gets that 
$(c_{n + 1} - c_n) \langle \tau, \varphi(x_n) - \varphi(x_{n + 1}) \rangle \ge 0$.
Therefore $\langle \tau, \varphi(x_n) - \varphi(x_{n + 1}) \rangle \ge 0$ for all $n \in \mathbb{N}$, since the sequence
$\{ c_n \}$ is strictly
increasing. Consequently, $f(x_n) \le f(x_{n + 1})$ due to \eqref{eq:ObjFuncDecay2}.

Arguing by reductio ad absurdum, suppose that the sequence $\{ \varphi(x_n) \}$ does not converge to zero. Then there
exist $\eta > 0$ and a subsequence $\{ x_{n_k} \}$ such that 
$\langle \tau, \varphi(x_{n_k}) \rangle \ge p_K(\tau) \| \varphi(x_{n_k}) \| \ge \eta$ for all $k \in \mathbb{N}$.
Hence taking into account the fact that the sequence $\{ f(x_n) \}$ is nondecreasing one obtains that
\[
  \Phi_{\tau_{n_k}}(x_{n_k}) = f(x_{n_k}) + \langle \tau_{n_k}, \varphi(x_{n_k}) \rangle
  \ge f(x_1) + c_{n_k} \eta \quad \forall k \in \mathbb{N}.
\]
Therefore $\Phi_{\tau_{n_k}}(x_{n_k}) \to + \infty$ as $k \to \infty$. On the other hand, for any feasible point $x$ of
the problem $(\mathcal{P})$ such that $f(x) < + \infty$ one has 
\[
  \Phi_{\tau_n}(x_n) \le \Phi_{\tau_n}(x) = f(x) < + \infty	\quad \forall n \in \mathbb{N},
\]
which contradicts the fact that $\Phi_{\tau_{n_k}}(x_{n_k}) \to + \infty$ as $k \to \infty$. Thus, $\varphi(x_n) \to 0$
as $n \to \infty$.

If the function $\| \varphi(\cdot) \|$ is l.s.c. on $Q$, and $x_*$ is a limit point of the sequence $\{ x_n \}$, then
obviously $\| \varphi(x_*) \| = 0$, i.e. $x_*$ is a feasible point of the problem $(\mathcal{P})$. Note that for any
globally optimal solution $z_*$ of the problem $(\mathcal{P})$ one has
\[
  f(z_*) = \Phi_{\tau_n}(z_*) \ge \Phi_{\tau_n}(x_n) \ge f(x_n) \quad \forall n \in \mathbb{N},
\]
that is, $f(x_n) \le f(z_*)$ for all $n \in \mathbb{N}$. Consequently, $f(x_*) \le f(z_*)$ (provided $f$ is l.s.c. on
$Q$), and $x_*$ is a globally optimal solution of the problem $(\mathcal{P})$.
\end{proof}

Consider the simplest (`naive') exact penalty method utilising the penalty function $\Phi_{\tau}$
(see~Algorithmic Pattern~\ref{alg:SimplestExPen}). Our first aim is to prove a natural convergence theorem for this
method, which
will serve as a foundation for our algorithmic approach to global exactness. 

\begin{algorithm}[ht!]	\label{alg:SimplestExPen}
\caption{The Simplest Global Exact Penalty Method}

\noindent\textbf{Initial data.} {Choose $\tau_1 \in \mathcal{T}$ and $\theta > 1$, and set $n := 1$.}

\noindent\textbf{Main Step.} {Set the value of $x_n$ to a globally optimal solution of the penalized problem
\[
  \text{minimize} \enspace \Phi_{\tau_n}(x) \quad \text{subject to} \quad x \in Q.
\]
If $n \ge 2$ and $\Phi_{\tau_n}(x_n) = \Phi_{\tau_{n - 1}}(x_{n - 1})$, \textbf{Stop}. Otherwise, put 
$\tau_{n + 1} = \theta \tau_n$, $n := n + 1$, and repeat the \textbf{Main Step}.
}
\end{algorithm}

Observe that if the penalty function $\Phi_{\tau}$ is globally exact, then it necessarily is locally exact at every
globally optimal solution of the problem $(\mathcal{P})$. Therefore, it is natural to analyse a behaviour of sequences
generated by Algorithmic Pattern~\ref{alg:SimplestExPen} under the assumptions that $\Phi_{\tau}$ is locally exact at
every
globally optimal solution of the problem $(\mathcal{P})$.

\begin{theorem} \label{thrm:SimplestExPenConvergence}
Let the set $Q$ be closed, the functions $f$ and $\| \varphi(\cdot) \|$ be l.s.c. on $Q$, and 
$\tau_1 \in \mathcal{T} \cap K^*_{++}$. Suppose also that the penalty function $\Phi_{\tau}$ is locally exact at every
globally optimal solution $x_*$ of the problem $(\mathcal{P})$. Then Algorithmic Pattern~\ref{alg:SimplestExPen} either
terminates
after a finite number of iterations by finding a globally optimal solution of the problem $(\mathcal{P})$ or generates
an infinite sequence $\{ x_n \}$ that has no limit points.
\end{theorem}

\begin{proof}
Let us first note that if Algorithmic Pattern~\ref{alg:SimplestExPen} terminates after a finite number of iterations,
then 
the last computed point $x_n$ is a globally optimal solution of the problem $(\mathcal{P})$. Indeed, suppose that 
$\Phi_{\tau_n}(x_n) = \Phi_{\tau_{n - 1}}(x_{n - 1})$ for some $n \in \mathbb{N}$. Recall that 
$\tau_n = \theta \tau_{n - 1}$ and $\theta > 1$. Therefore for any point $x$ that is infeasible for the problem
$(\mathcal{P})$ (i.e. $\varphi(x) \in K \setminus \{ 0 \}$) one has
\begin{align*}
  \Phi_{\tau_n}(x) = f(x) + \theta \langle \tau_{n - 1}, \varphi(x) \rangle 
  &> f(x) + \langle \tau_{n - 1}, \varphi(x) \rangle 
  \\
  &= \Phi_{\tau_{n - 1}}(x) \ge \Phi_{\tau_{n - 1}}(x_{n - 1}),
\end{align*}
which implies that the point $x_n$ is feasible for the problem $(\mathcal{P})$. Hence taking into account the
fact that $\Phi_{\tau}(x) = f(x)$ for any feasible point $x$ and any $\tau \in K^*$, one obtains that $x_n$ is a
globally optimal solution of the problem $(\mathcal{P})$.

Now we turn to the proof of the main statement of the theorem. Arguing by reductio ad absurdum, suppose that
Algorithmic Pattern~\ref{alg:SimplestExPen} does not terminate after a finite number of iterations and generates a
sequence 
$\{ x_n \}$ that has a limit point $x_*$. Then there exists a subsequence $\{ x_{n_k} \}$ converging to $x_*$. By
Lemma~\ref{lemma:ExPenMinimizingSeq} the point $x_*$ is a globally optimal solution of the problem $(\mathcal{P})$.
By our assumption the penalty function $\Phi_{\tau}$ is locally exact at $x_*$. Therefore by
Remark~\ref{rmrk:PenaltyParameterMultiple} there exists $c(x_*) > 0$ such that the vector $c(x_*) \tau_1$ is a local
exact penalty parameter at $x_*$, which implies that there exists a neighbourhood $U$ of $x_*$ such that
\[
  \Phi_{c \tau_1}(x) \ge \Phi_{c \tau_1}(x_*) = f(x_*) =: f_* \quad \forall x \in U \cap Q, \: c \ge c(x_*).
\]
By definitions $\tau_n = \theta^n \tau_1$ and $\{ x_{n_k} \}$ converges to $x_*$. Therefore there exists 
$k \in \mathbb{N}$ such that $x_{n_k} \in U$ and $\theta^{n_k} \ge c(x_*)$. For any such $k$ one has
\[
  \Phi_{\tau_{n_k}}(x_{n_k}) \ge \Phi_{\tau_{n_k}}(x_*) = f_*,
\]
that is, $\Phi_{\tau_{n_k}}(x_{n_k}) = f_*$. Hence bearing in mind the facts that for all $x \in X$ one has
$\Phi_{\tau_{n + 1}}(x) \ge \Phi_{\tau_n}(x)$, and $\Phi_{\tau_n}(x_*) = f_*$ for all $n \in \mathbb{N}$, one obtains
that
\[
  f_* \ge \Phi_{\tau_{n_k + 1}}(x_{n_k + 1}) \ge \Phi_{\tau_{n_k}}(x_{n_k}) = f_* \quad \forall n \ge n_k,
\]
which contradicts our assumption that Algorithmic Pattern~\ref{alg:SimplestExPen} does not terminate after a finite
number of iterations.
\end{proof}

As a straightforward corollary to the previous theorem, we can obtain simple necessary and sufficient conditions for
the global exactness of the penalty function $\Phi_{\tau}$.

\begin{theorem} \label{thrm:AlgorithmicGlobalExact}
Let $\mathcal{T} \subseteq K^*_{++}$. Then the penalty function $\Phi_{\tau}$ is globally exact if and only if 
the two following conditions hold true:
\begin{enumerate}
\item{$\Phi_{\tau}$ is locally exact at every globally optimal solution of the problem $(\mathcal{P})$;
}

\item{Algorithmic Pattern~\ref{alg:SimplestExPen} with arbitrary $\tau_1 \in \mathcal{T}$ terminates after a finite
number of
iterations.
}
\end{enumerate}
\end{theorem}

\begin{proof}
Let $\Phi_{\tau}$ be globally exact with exact penalty parameter $\tau_*$. Then, obviously, $\Phi_{\tau}$ is locally
exact at every globally optimal solution of the problem $(\mathcal{P})$ with the same exact penalty parameter.

By Remark~\ref{rmrk:PenaltyParameterMultiple} there exists $c_0 > 0$ such that $c_0 \tau_1$ is a global exact penalty
parameter of $\Phi_{\tau}$. Clearly, $\theta^n \ge c_0$ for some $n \in \mathbb{N}$, which implies that
$\tau_n$ is a global exact penalty parameter of $\Phi_{\tau}$. Consequently, Algorithmic Pattern~\ref{alg:SimplestExPen}
terminates after at most $n + 1$ iterations, since by the definition of global exactness the points $x_n$ and
$x_{n + 1}$ are globally optimal solutions of the problem $(\mathcal{P})$ and
$\Phi_{\tau_n}(x_n) = \Phi_{\tau_{n + 1}}(x_{n + 1}) = f_*$.

Let us prove the converse statement. If Algorithmic Pattern~\ref{alg:SimplestExPen} with 
$\tau_1 \in \mathcal{T} \subseteq K^*_{++}$ terminates after a finite number of iterations, then by
Theorem~\ref{thrm:SimplestExPenConvergence} the last computed point $x_{n + 1}$ is a global minimizer of the
problem $(\mathcal{P})$, which obviously implies that the penalty function $\Phi_{\tau}$ is globally exact.
\end{proof}

\begin{remark}
The previous theorem can be restated as follows. Let $\Phi_{\tau}$ be locally exact at every globally optimal solution
of the problem $(\mathcal{P})$. Then $\Phi_{\tau}$ is \textit{not} globally exact iff a sequence generated by
Algorithmic Pattern~\ref{alg:SimplestExPen} has \text{no} limit points. Thus, the global exactness of the penalty
function
$\Phi_{\tau}$ is completely predefined by its behaviour near globally optimal solutions of the problem $(\mathcal{P})$
and behaviour of sequences generated by global exact penalty methods. Moreover, it seems more natural to study global
exactness
of penalty functions in the context of exact penalty methods than on its own, since this way one can avoid restrictive
assumptions on constraints. In particular, in Theorem~\ref{thrm:AlgorithmicGlobalExact} we do not impose any nonlocal
assumptions on constraints, such as the Palais-Smale conditions from \cite{Zaslavski,Zaslavski2013} or an assumption
from
\cite{DolgopolikFominyh,Dolgopolik_Control} that ensures uniform nonlocal metric regularity of constraints.
\end{remark}

The penalty updating strategy from Algorithmic Pattern~\ref{alg:SimplestExPen} ($\tau_{n + 1} = \theta \tau_n$ for some
fixed
$\theta > 1$) largely negates the benefits of using vectorial penalty parameter. Instead of adjusting the penalty
parameter adaptively, i.e. in a way that takes into account which constraints have greater violation measure, we simply
increase the penalty parameter by a constant factor. To overcome this issue, let us present and analyse a modified
version of Algorithmic Pattern~\ref{alg:SimplestExPen} with adaptive penalty updates, largely inspired by the penalty
updates
from paper \cite{BurachikKayaPrice}, in which a primal-dual approach to penalty updating was considered.

Suppose that there is an embedding $i \colon K \to K^*$ and $\mathcal{T} + i(K) \subseteq \mathcal{T}$, i.e. the set
$\mathcal{T}$ is closed under addition with vectors from the set $i(K)$. A theoretical scheme of exact penalty method
with adaptive penalty updates is given in Algorithmic Pattern~\ref{alg:ExPenAdaptive}. Let us note that the simplest
choice of 
the scaling parameters $s_n$ is $s_n \equiv 1$. One can also set $s_n = \gamma / \| \varphi(x_n) \|$ for some 
$\gamma > 0$ to avoid an excessive increase of the norm of the penalty parameter $\tau_n$, when the infeasibility
measure 
$\| \varphi(x_n) \|$ is sufficiently large.

\begin{algorithm}[ht!]	\label{alg:ExPenAdaptive}
\caption{Global Exact Penalty Method with Adaptive Penalty Updates}

\noindent\textbf{Initial data.} {Choose $\tau_1 \in \mathcal{T}$, and set $n := 1$.}

\noindent\textbf{Main Step.} {Set the value of $x_n$ to a globally optimal solution of the penalized problem
\[
  \text{minimize} \enspace \Phi_{\tau_n}(x) \quad \text{subject to} \quad x \in Q.
\]
If $x_n$ is feasible for the problem $(\mathcal{P})$, \textbf{Stop}. Otherwise, choose a scaling coefficient $s_n > 0$,
put $\tau_{n + 1} = \tau_n + s_n i(\varphi(x_n))$ and $n := n + 1$, and repeat the 
\textbf{Main Step}.
}
\end{algorithm}

Observe that for the penalty function from Example~\ref{example:MathProgram} the penalty updates from
Algorithmic Pattern~\ref{alg:ExPenAdaptive} take the form
\begin{multline*}
  \tau_{n + 1} = \tau_n \\
  + s_n 
  \Big( \max\{ 0, g_1(x_n) \}, \ldots, \max\{ 0, g_{m_1}(x_n) \}, |g_{m_1 + 1}(x_n)|, \ldots, |g_{m_2}(x_n)| \Big)^T,
\end{multline*}
and, in essence, coincide with the penalty updates from \cite{BurachikKayaPrice}. In this case the increase of each
coordinate of the penalty parameter is proportional to the violation of the corresponding constraint. Those components
of the penalty parameter that correspond to constraints with greater violation are increased more, while if a certain
constraint is ``almost satisfied'', then then the corresponding penalty parameter is changed only slightly. Note also
that in the case of the exact penalty function for nonlinear semidefinite programming problems from
Example~\ref{example:ConeConstrained}, the penalty updates take the form $\tau_{n + 1} = \tau_n + s_n [G(x_n)]_+$.
Finally, for the penalty function from Example~\ref{example:OptimalControl}  the penalty updates take the form $\tau_{n
+ 1}(t) = \tau_n(t) + s_n \max\{ 0, g(x_n(t), t) \}$, $t \in [0, T]$. In this case, the penalty parameter is increased
more at those points $t \in [0, T]$ for which the violation of the state constraint is greater.

\begin{remark}
Let us note that the penalty updating rule $\tau_{n + 1} = \tau_n + s_n i(\varphi(x_n))$ no longer allows one to use
the equality $\Phi_{\tau_n}(x_n) = \Phi_{\tau_{n - 1}}(x_{n - 1})$ as a stopping criterion, since the validity of this
equality no longer implies that the point $x_n$ is feasible for the problem $(\mathcal{P})$ (unless 
$i(\varphi(x_n)) \in K^*_+$). Therefore, we chose the feasibility of the point $x_n$ as a termination criterion. 
A more practical stopping rule would be $\| \varphi(x_n) \| < \varepsilon$ for some small $\varepsilon > 0$.
\end{remark}

Let us present a convergence theorem for Algorithmic Pattern~\ref{alg:ExPenAdaptive} in the case when $\Phi_{\tau}$ is a
penalty
function from Example~\ref{example:MathProgram}. A convergence analysis of Algorithmic Pattern~\ref{alg:ExPenAdaptive}
in the
general case remains a challenging open problem.

\begin{theorem} \label{thrm:AdaptiveExPenConvergence}
Let $Y = \mathbb{R}^m$, $K = \mathbb{R}^m_+$ be the nonnegative orthant, and $\mathcal{T} \subset \mathbb{R}^m_+$ be 
the set of vectors with positive coordinates. Let also the set $Q$ be closed, the functions $f$ and 
$\| \varphi(\cdot) \|$ be l.s.c. on $Q$, and there exist $\gamma > 0$ such that either $s_n \ge \gamma > 0$ for all $n
\in \mathbb{N}$ or $s_n \ge \gamma / \| \varphi(x_n) \|$ for all $n \in \mathbb{N}$. Then 
Algorithmic Pattern~\ref{alg:ExPenAdaptive} either terminates after a finite number of iterations by finding a globally
optimal
solution of the problem $(\mathcal{P})$ or generates an infinite sequence $\{ x_n \}$ such that $\varphi(x_n) \to 0$
as $n \to \infty$, and all limit points of the sequence $\{ x_n \}$ (if exist) are globally optimal solutions of the
problem $(\mathcal{P})$. Moreover, if the sequence of penalty parameters $\{ \tau_n \}$ is bounded and there exists a
limit point of the sequence $\{ x_n \}$, then the penalty function $\Phi_{\tau}$ is globally exact, and the sequence 
$\{ \tau_n \}$ converges to a point $\tau_*$ such that any $\tau \succeq_{K^*_+} \tau_*$ is a global exact penalty
parameter.
\end{theorem}

\begin{proof}
Let us first note that if the algorithmic pattern terminates after a finite number of iterations, then the last computed
point
$x_n$ is a globally optimal solution of the problem $(\mathcal{P})$. Indeed, by definition
Algorithmic Pattern~\ref{alg:ExPenAdaptive} terminates, if the point $x_n$ is feasible for the problem $(\mathcal{P})$.
Hence bearing in mind the facts that $x_n$ is a global minimizer of $\Phi_{\tau_n}$ on the set $Q$, and
$\Phi_{\tau}(x) = f(x)$ for any feasible point of the problem $(\mathcal{P})$, one gets that $x_n$ is a global minimizer
of the problem $(\mathcal{P})$.

Suppose now that Algorithmic Pattern~\ref{alg:ExPenAdaptive} generates an infinite sequence $\{ x_n \}$. Let us verify
that
$\varphi(x_n) \to 0$ as $n \to \infty$. Indeed, arguing by reductio ad absurdum, suppose that the sequence 
$\{ \varphi(x_n) \}$ does not converge to zero. We consider two cases.

\textbf{Case I.} Let $s_n \ge \gamma > 0$ for all $n \in \mathbb{N}$. Since the sequence $\{ \varphi(x_n) \}$ does not
converge to zero, there exist $\varepsilon > 0$, a subsequence $\{ \varphi(x_{n_k}) \}$, and an index 
$i \in \{ 1, \ldots, m \}$ such that $\varphi^{(i)}(x_{n_k}) \ge \varepsilon$ for all $k \in \mathbb{N}$, where
$\varphi^{(i)}(x)$ is the $i$-th coordinate of the vector $\varphi(x) \in \mathbb{R}^m$. Then according to
Algorithmic Pattern~\ref{alg:ExPenAdaptive} for all $k \in \mathbb{N}$ one has
\[
  \tau_{n_k + s}^{(i)} \ge \tau_1^{(i)} + k \gamma \varepsilon, \quad
  \tau_{n_k + s}^{(j)} \ge \tau_1^{(j)} \quad \forall j \ne i, \enspace \forall s \in \mathbb{N},
\]
which yields
\[
  \Phi_{\tau_{n_k}}(x_{n_k}) \ge \Phi_{\tau_1}(x_{n_k}) + (k - 1) \gamma \varepsilon \varphi^{(i)}(x_{n_k})
  \ge \Phi_{\tau_1}(x_1) + (k - 1) \gamma \varepsilon^2
\]
for all $k \in \mathbb{N}$. Therefore $\Phi_{\tau_{n_k}}(x_{n_k}) \to + \infty$ as $k \to \infty$, which contradicts the
fact that
\[
  f_* \ge \Phi_{\tau_{n_k}}(x_*) \ge \Phi_{\tau_{n_k}}(x_{n_k}) \quad \forall k \in \mathbb{N}
\]
for any globally optimal solution $x_*$ of the problem $(\mathcal{P})$.

\textbf{Case II.} Let $s_n \ge \gamma / \| \varphi(x_n) \| > 0$ for all $n \in \mathbb{N}$. By our assumption there
exist $\varepsilon > 0$ and a subsequence $\{ x_{n_k} \}$ such that $\| \varphi(x_{n_k}) \| \ge \varepsilon$. Since
$\varphi(x)$ is an $m$-dimensional vector with non-negative coordinates, there exist a subsequence, which we denote
again by $\{ x_{n_k} \}$, and an index $i \in \{ 1, \ldots, m \}$ such that $\varphi^{(i)}(x_{n_k})$ is the greatest
coordinate of the vector $\{ \varphi(x_{n_k}) \}$ for any $k \in \mathbb{N}$. 

Let $C > 0$ be such that $\| y \| \le C \| y \|_{\infty}$ for all $y \in \mathbb{R}^m$. Then bearing in mind the fact
that by definition $\varphi^{(i)}(x_{n_k}) = \| \varphi(x_{n_k}) \|_{\infty}$ one obtains that
$\varphi^{(i)}(x_{n_k}) / \| \varphi(x_{n_k}) \| \ge 1 / C$ for all $k \in \mathbb{N}$, which yields
\[
  \tau_{n_k + s}^{(i)} \ge \tau_1^{(i)} + k \frac{\gamma}{C}, \quad
  \tau_{n_k + s}^{(j)} \ge \tau_1^{(j)} \quad \forall j \ne i, \enspace \forall s \in \mathbb{N}.
\]
Hence with the use of the inequality 
\[
  \varphi^{(i)}(x_{n_k}) = \| \varphi(x_{n_k}) \|_{\infty} \ge \frac{1}{C} \| \varphi(x_{n_k}) \|
  \ge \frac{\varepsilon}{C}  \quad \forall k \in \mathbb{N}
\]
one obtains that for all $k \in \mathbb{N}$ the following inequalities hold true
\[
  \Phi_{\tau_{n_k}}(x_{n_k}) \ge \Phi_{\tau_1}(x_{n_k}) + (k - 1) \frac{\gamma}{C} \varphi^{(i)}(x_{n_k})
  \ge \Phi_{\tau_1}(x_1) + (k - 1) \gamma \frac{\varepsilon}{C^2},
\]
which just like in the first case leads to an obvious contradiction.

Let us now check that limit points of the sequence $\{ x_n \}$ are globally optimal solutions of 
the problem $(\mathcal{P})$. Indeed, let $x_*$ be a limit point of this sequence, i.e. there exists a subsequence 
$\{ x_{n_k} \}$ converging to $x_*$. Since the set $Q$ is closed and the function $\| \varphi(\cdot) \|$ is l.s.c. on
$Q$, one obtains that $x_*$ is a feasible point of the problem $(\mathcal{P})$ (recall that $\varphi(x) = 0$ iff 
$x \in M$). 

Observe that for any globally optimal solution $z_*$ of the problem $(\mathcal{P})$ and for all $n \in \mathbb{N}$ one
has
\[
  f_* := f(z_*) = \Phi_{\tau_n}(z_*) \ge \min_{x \in Q} \Phi_{\tau_n}(x) = \Phi_{\tau_n}(x_n) \ge f(x_n).
\]
Therefore $f(x_{n_k}) \le f_*$ for all $k \in \mathbb{N}$. Passing to the limit as $k \to \infty$, one gets 
$f(x_*) \le f_*$, that is, $x_*$ is a globally optimal solution of the problem $(\mathcal{P})$.

Suppose, finally, that the sequence $\{ \tau_n \}$ is bounded. By definition each coordinate $\tau_n^{(i)}$ is
nondecreasing in $n$. Therefore the sequence $\{ \tau_n \}$ converges to some $\tau_*$. Arguing by reductio ad absurdum,
suppose that there exists a vector $\tau \succeq_{K^*_+} \tau_*$ that is not a global exact penalty parameter.
Then by Lemma~\ref{lemma:ExactnessViaOptValue} one has $f_* > \inf_{x \in Q} \Phi_{\tau_*}(x)$. Hence taking into
account the fact that the sequence $\{ \tau_n \}$ is coordinate-wise nondecreasing one obtains that
\[
  f_* > \inf_{x \in Q} \Phi_{\tau_*}(x) \ge \inf_{x \in Q} \Phi_{\tau_n}(x) = \Phi_{\tau_n}(x_n) = f(x_n)
\]
for all $n \in \mathbb{N}$. Therefore $\limsup_{n \to \infty} f(x_n) < f_*$, which contradicts the fact that limit
points of the sequence $\{ x_n \}$, which exist by our assumption, are globally optimal solutions of the problem
$(\mathcal{P})$. Thus, any $\tau \succeq_{K^*_+} \tau_*$ is a global exact penalty parameter, and the penalty function
$\Phi_{\tau}$ is globally exact.
\end{proof}

\begin{remark}
From the theorem above it follows that if the penalty function $\Phi_{\tau}$ is \textit{not} globally exact, then
either a sequence $\{ x_n \}$ generated by Algorithmic Pattern~\ref{alg:ExPenAdaptive} has no limit points or the
corresponding
sequence of penalty parameters $\{ \tau_n \}$ is unbounded. Thus, if the sequence $\{ x_n \}$ has limit points, then the
global exactness of the penalty function $\Phi_{\tau}$ is a necessary condition for the boundedness of the sequence of
penalty parameters. Whether this condition is also sufficient for the boundedness of the sequence $\{ \tau_n \}$ is an
open problem.
\end{remark}

Let us finally show that under some natural assumptions a simple combination of the penalty updates from
Algorithmic Patterns~\ref{alg:SimplestExPen} and \ref{alg:ExPenAdaptive}, on the one hand, guarantees a finite
convergence to a
globally optimal solution of the problem $(\mathcal{P})$, but on the other hand, preserves all practical benefits of
the adaptive penalty updates from Algorithmic Patterns~\ref{alg:ExPenAdaptive}. Furthermore, one can prove the finite
convergence
of the algorithmic pattern with combined penalty updates (see Algorithmic Pattern~\ref{alg:ExPenCombined}) without
imposing any assumptions
on the space $Y$ and the cone $K$.

\begin{algorithm}[ht!]	\label{alg:ExPenCombined}
\caption{Global Exact Penalty Method with Combined Penalty Updates}

\noindent\textbf{Initial data.} {Choose $\tau_1 \in \mathcal{T}$ and some small $\delta > 0$, and set $n := 0$.}

\noindent\textbf{Main Step.} {Set the value of $x_n$ to a globally optimal solution of the penalized problem
\[
  \text{minimize} \enspace \Phi_{\tau_n}(x) \quad \text{subject to} \quad x \in Q.
\]
If $x_n$ is feasible for the problem $(\mathcal{P})$, \textbf{Stop}. Otherwise, choose a scaling coefficient $s_n > 0$,
put $\tau_{n + 1} = \tau_n + \delta \tau_1 + s_n i(\varphi(x_n))$ and $n := n + 1$, and repeat the \textbf{Main Step}.
}
\end{algorithm}

Let us note that penalty updates similar to the ones used in Algorithmic Pattern~\ref{alg:ExPenCombined} were studied in
\cite{BurachikKayaPrice} for a smoothing approximation of an $\ell_1$ penalty function for nonlinear programming
problems.

\begin{theorem}
Let the set $Q$ be closed, the functions $f$ and $\| \varphi(\cdot) \|$ be l.s.c. on $Q$, and 
$\tau_1 \in \mathcal{T} \cap K^*_{++}$. Suppose also that the penalty function $\Phi_{\tau}$ is locally exact at every
globally optimal solution $x_*$ of the problem $(\mathcal{P})$. Then Algorithmic Pattern~\ref{alg:ExPenCombined} either
terminates after a finite number of iterations by finding a globally optimal solution of the problem $(\mathcal{P})$
or generates an infinite sequence $\{ x_n \}$ that has no limit points.
\end{theorem}

\begin{proof}
Arguing in the same way as in the proof of Theorem~\ref{thrm:AdaptiveExPenConvergence}, one can easily check that if the
method terminates after a finite number of iterations, then the last computed point is a globally optimal solution
of the problem $(\mathcal{P})$. Therefore, arguing by reductio ad absurdum, suppose that
Algorithmic Pattern~\ref{alg:ExPenCombined} generates an infinite sequence $\{ x_n \}$ that has a limit point $x_*$.

Let us check that $\varphi(x_n) \to 0$ as $n \to \infty$. Indeed, according to the penalty updating rule from
Algorithmic Pattern~\ref{alg:ExPenCombined} for any $n \in \mathbb{N}$ one has
\begin{multline*}
  \Phi_{\tau_{n + 1}}(x_{n + 1}) \ge f(x_{n + 1}) + (1 + n \delta) \langle \tau_1, \varphi(x_{n + 1}) \rangle
  \\
  \ge \Phi_{\tau_1}(x_{n + 1}) + n \delta p_K(\tau_1) \| \varphi(x_{n + 1}) \|
  \ge \Phi_{\tau_1}(x_1) + n \delta p_K(\tau_1) \| \varphi(x_{n + 1}) \|.
\end{multline*}
Consequently, if the sequence $\{ \varphi(x_n) \}$ does not converge to zero, then one has
$\limsup_{n \to \infty} \Phi_{\tau_n}(x_n) = + \infty$, which contradicts the fact that
\[
  f(x_*) = \Phi_{\tau_n}(x_*) \ge \Phi_{\tau_n}(x_n) \quad \forall n \in \mathbb{N}
\]
for any globally optimal solution $x_*$ of the problem $(\mathcal{P})$.

Utilising the fact that $\varphi(x_n) \to 0$ as $n \to \infty$ and arguing in the same way as in the proof of 
Theorems~\ref{thrm:SimplestExPenConvergence} and \ref{thrm:AdaptiveExPenConvergence} one can check that $x_*$ is a
globally optimal solution of the problem $(\mathcal{P})$. Therefore, by our assumption the penalty function
$\Phi_{\tau}$ is locally exact at $x_*$, while by Remark~\ref{rmrk:PenaltyParameterMultiple} there exists $c(x_*) > 0$
such that $c(x_*) \tau_1$ is a local exact penalty parameter of $\Phi_{\tau}$ at $x_*$. Consequently, there exists a
neighbourhood $U$ of $x_*$ such that
\[
  \Phi_{\tau}(x) \ge \Phi_{\tau}(x_*) = f(x_*) \quad 
  \forall x \in U \cap Q, \enspace \tau \succeq_{K^*} c(x_*) \tau_1.
\]
Observe that according to Algorithmic Pattern~\ref{alg:ExPenCombined} one has 
\[
  \tau_{n + 1} = (1 + n \delta) \tau_1 + \sum_{i = 1}^n s_n i(\varphi(x_n)) \quad \forall n \in \mathbb{N}
\]
and by our assumption $i(\varphi(x_n)) \in K^*$. Therefore there exists $n_0 \in \mathbb{N}$ such that $\tau_n
\succeq_{K^*} c(x_*) \tau_1$ for all $n \ge n_0$ ($n_0$ must satisfy the inequality $1 + n_0 \delta \ge c(x_*)$).
Moreover, from the fact that $x_*$ is a limit point of
the sequence $\{ x_n \}$ it follows that $x_n \in U$ for some $n \ge n_0$. Consequently, for any such $n$ one has
\[
  \min_{x \in Q} \Phi_{\tau_n}(x) =: \Phi_{\tau_n}(x_n) \ge \Phi_{\tau_n}(x_*) = f_*.
\]
Recall that $\tau_1 \in K^*_{++}$. Therefore 
\[
  \Phi_{\tau_{n + 1}}(x) = \Phi_{\tau_n}(x) + \delta \langle \tau_1, \varphi(x) \rangle + 
  s_n \langle i(\varphi(x_n)), \varphi(x) \rangle > \Phi_{\tau_n}(x) \ge f_*
\]
for any point $x$ that is infeasible for the problem $(\mathcal{P})$. On the other hand, 
$\Phi_{\tau_{n + 1}}(x_*) = f_*$ for any globally optimal solution $x_*$ of the problem $(\mathcal{P})$. Therefore, the
point
$x_{n + 1}$ must be a globally optimal solution of the problem $(\mathcal{P})$, which contradicts our assumption that 
Algorithmic Pattern~\ref{alg:ExPenCombined} does not terminate after a finite number of iterations.
\end{proof}

\begin{remark}
Let us note that the finite convergence to a globally optimal solution of the problem $(\mathcal{P})$ can be proved, if
the penalty updates from Algorithmic Pattern~\ref{alg:ExPenAdaptive} are corrected as in Algorithmic
Pattern~\ref{alg:ExPenCombined} only
once every certain number of iterations, that is, the following penalty updates are used
\[
  \tau_{n + 1} = \begin{cases}
    \tau_n + \delta \tau_1 + s_n i(\varphi(x_n), & \text{ if } n = k \ell \text{ for some } k \in \mathbb{N},
    \\
    \tau_n + s_n i(\varphi(x_n)), & \text{ otherwise.}
  \end{cases}
\]
for some fixed $\ell \in \mathbb{N}$. In other words, one adds a small correction $\delta \tau_1$ to the penalty
updates from Algorithmic Pattern~\ref{alg:ExPenAdaptive} only every $\ell$ iterations. Then one has
$\Phi_{\tau_{n + 1}}(x_{n + 1}) \ge f(x_{n + 1}) + (1 + k \delta) \langle \tau_1, \varphi(x_{n + 1}) \rangle$ for any 
$n, k \in \mathbb{N}$ such that $n \ge k \ell$. Arguing in the same way as in the proof of the previous theorem
and applying the inequality above, one can easily prove that the corresponding algorithmic pattern either terminates
after a
finite number of iterations by finding globally optimal solution of the problem $(\mathcal{P})$ or generates an
infinite sequence that has no limit points.
\end{remark}

\begin{remark}
Although in this section we analysed penalty updating strategies only in the context of global exactness and
corresponding global optimization methods (cf.~\cite{DiPilloLucidiRinaldi}), which cannot be implemented and applied to 
practical problems directly, such penalty updating strategies can be applied to local optimization methods as well.
Namely, one can apply penalty updates from Algorithmic Patterns~\ref{alg:ExPenAdaptive} and \ref{alg:ExPenCombined}
after each iteration of a local optimization method using exact penalty functions. Alternatively, instead of computing
$x_{n + 1}$ in Algorithmic Patterns~\ref{alg:ExPenAdaptive} and \ref{alg:ExPenCombined} as a global minimizer of the
penalty function, one can define $x_{n + 1}$ as a local minimizer/stationary point of this function. As is demonstrated
by numerical experiments in \cite{BurachikKayaPrice}, even such naive ``local'' implementation of \textit{global} exact
penalty methods significantly outperforms some state-of-the-art optimization solvers, both in terms of computation time
and quality of computed local minimizers. Thus, although the algorithmic patterns discussed in this paper cannot be
implemented and applied directly, they can be used as a foundation for the development of efficient optimization
methods.
\end{remark}


\bibliographystyle{abbrv}  
\bibliography{Dolgopolik_bibl}

\end{document}